\documentclass[12pt]{scrartcl}

\usepackage{amsmath,amssymb,amsthm}
\usepackage{url}
\usepackage{varioref}
\usepackage{cite}
\usepackage{fancyvrb}
\DefineVerbatimEnvironment{sage}{LVerbatim}{numbers=left,numbersep=2mm,frame=lines,framerule=0.3mm,label=Sage,fontsize=\small}

\usepackage{authblk}

\newtheorem{theorem}{Theorem}[section]
\newtheorem{corollary}[theorem]{Corollary}
\theoremstyle{remark}
\newtheorem{remark}[theorem]{Remark}

\title{Closed-form expansions for the universal edge
  elimination polynomial}
\author{K. Dohmen} 
\affil{Department of Mathematics, %
Mittweida University of Applied Sciences, Germany \\ %
Electronic address: dohmen@hs-mittweida.de}

\newcommand{\poly}[1]{\xi(#1,x,y,z)}

\begin{document}

\maketitle

\begin{abstract}
\emph{Abstract.} We establish closed-form expansions for the universal edge
elimination polynomial of paths and cycles and their generating
functions. This includes closed-form expansions for the bivariate matching
polynomial, the bivariate chromatic polynomial, and the covered components
polynomial.
\end{abstract}

\section{Introduction}

As a generalization of several well-known graph polynomials, Averbouch, Godlin
and Makowsky \cite{AGM:2010} introduced the so-called \emph{universal edge
  elimination polynomial} $\poly{G}$, whose recursive definition involves
three kinds of edge elimination:
\begin{description}
\item[\qquad $G_{-e}$:] The graph obtained from $G$ by removing the edge $e$.
\item[\qquad $G_{/e}$:] The graph obtained from $G$ by removing $e$ and
identifying its endpoints,
\item[\qquad $G_{\dagger e}$:] The graph obtained from $G$ by removing $e$ and
all incident vertices.
\end{description}                       
All graphs are considered as finite and undirected, and may have loops and
multiple edges.  We use $P_n$ to denote the simple path with $n$ vertices
($n=0,1,\dots)$, and $\oplus$ to denote the disjoint union of graphs.
According to \cite{AGM:2010}, $\poly{G}$ is defined by
\begin{gather}
\poly{P_0} = 1, \quad \poly{P_1} = x, \label{eq:14} \\
\poly{G} = \poly{G_{-e}} + y \poly{G_{/e}} + z \poly{G_{\dagger e}} , \label{eq:rek} \\
\poly{G_1 \oplus G_2} = \poly{G_1} \poly{G_2} . \label{eq:rekk}
\end{gather}
The universal edge elimination polynomial $\poly{G}$ generalizes, among
others, the bivariate matching polynomial $M(G,x,y) = \xi(G,x,0,y)$ (provided
$G$ is loop-free), the bivariate chromatic polynomial
$P(G,x,y)=\xi(G,x,-1,x-y)$, and the covered components polynomial $C(G,x,y,z)=
\xi(G,x,y,xyz-xy)$.  The implications of our results on $\poly{G}$ for these
polynomials are new as well.  We refer to
\cite{AGM:2010,DPT:2003,GH:1983,Tri:2011} for the definitions of the various
graph polynomials and the relationships among them.

\section{Closed-form expansions for paths and cycles}
\label{sec:clos-form-expans}

We use $\mathbb{N}$ to denote the set of positive integers. The following
theorem provides a closed-form expansion for the universal edge elimination
polynomial of a path.

\begin{theorem}
\label{thm1}
Let $n\in\mathbb{N}$, and $x,y,z\in \mathbb{R}$.  If $z > - \left(
\frac{x+y}{2} \right)^2$, then
\begin{gather}
\poly{P_n} = \frac{\sqrt{D}-x+y}{2\sqrt{D}} \left( \frac{x+y-\sqrt{D}}{2}
\right)^n + \frac{\sqrt{D}+x-y}{2\sqrt{D}}\left( \frac{x+y+\sqrt{D}}{2}
\right)^n 
\label{eq:2} 
\end{gather}
where 
\begin{gather}
\label{eq:3}
D := x^2+2xy+y^2+4z.
\end{gather}
If $z < - \left( \frac{x+y}{2} \right)^2$, then
\begin{gather}
\poly{P_n} = 
(-z)^{n/2} \left( \cos(n\varphi) + \frac{x-y}{\sqrt{-D}} \sin(n\varphi)
\right) 
\label{eq:15}
\end{gather}
where 
\begin{gather}
\label{eq:17}
\varphi = \begin{cases} \arctan \frac{\sqrt{-D}}{x+y} & \text{if $x+y > 0$}, \\
\pi/2 & \text{if $x+y=0$}, \\
\pi+\arctan \frac{\sqrt{-D}}{x+y} & \text{if $x+y < 0$}.
\end{cases}
\end{gather}
If $z = - \left( \frac{x+y}{2} \right)^2$, then
\begin{gather}
\label{eq:8}
\poly{P_n} = \frac{(n+1)x - (n-1)y}{2}\left( \frac{x+y}{2} \right)^{n-1}.
\end{gather}
\end{theorem}

\begin{proof}
By choosing $e$ as an end edge of $P_n$, Eqs.~\eqref{eq:rek} and
\eqref{eq:rekk} yield the recurrence
\begin{gather}
\label{eq:4}
\poly{P_n} = (x+y)\poly{P_{n-1}} + z \poly{P_{n-2}} \quad (n\ge 2),
\end{gather}
where the initial conditions are given by Eq.~\eqref{eq:14}.  This is a
homogeneous linear recurrence of degree 2 with constant coefficients.  We
solve this recurrence by applying the method of characteristic roots.  The
characteristic equation of the recurrence is
\begin{gather}
\label{eq:16}
r^2-(x+y)r-z=0 , 
\end{gather}
with discriminant $D$, given by Eq.~\eqref{eq:3}.  In our three cases, we have
$D>0$, $D<0$, and $D=0$, respectively.  In the first two cases, the solution
to Eq.~\eqref{eq:4} is of the form
\begin{gather}
\label{eq:5}
\poly{P_n} = c_1 r_1^n + c_2 r_2^n 
\end{gather}
where $r_1,r_2$ are the distinct roots of Eq.~\eqref{eq:16} and $c_1,c_2$ are
chosen to satisfy Eq.~\eqref{eq:14}.  In the first case we have
\begin{align}
\label{eq:23}
\begin{aligned}
r_1 & = \frac{x+y - \sqrt{D}}{2}, \qquad\quad\,\,\,\, & c_1 & = \frac{\sqrt{D}- x +
  y}{2\sqrt{D}}, \\
r_2 & = \frac{x+y + \sqrt{D}}{2},  & c_2 & = \frac{\sqrt{D}+ x -
  y}{2\sqrt{D}}, 
\end{aligned}
\intertext{and in the second case,}
\label{eq:27}
\begin{aligned}
r_1 & = \frac{x+y}{2} - \frac{\sqrt{-D}}{2} \,i, \qquad & c_1 & = \frac12 +
\frac{x-y}{2\sqrt{-D}} \,i, \\
r_2 & = \frac{x+y}{2} + \frac{\sqrt{-D}}{2} \,i, &  c_2 & = \frac12 -
\frac{x-y}{2\sqrt{-D}} \, i. 
\end{aligned}
\end{align}
A little bit of extra work is needed in the second case in order to get rid of
the imaginary parts: Representing $r_1$ and $r_2$ in polar form and applying
Euler's formula we obtain
\begin{align}
\label{eq:25}
\begin{aligned}
r_1^n & = \left( \sqrt{-z} \,e^{-i\varphi} \right)^n \!\!\!\! & = (-z)^{n/2} \left(
\cos(n\varphi) -
\sin(n\varphi) i \right), \\
r_2^n & = \left( \sqrt{-z} \,e^{i\varphi} \right)^n & = (-z)^{n/2} \left(
\cos(n\varphi) + \sin(n\varphi) i \right),
\end{aligned}
\end{align} 
with $\varphi$ as in Eq.~\eqref{eq:17}.  Thus, Eq.~\eqref{eq:5} becomes
\begin{multline*}
\poly{P_n} = (-z)^{n/2} \left( \frac12 \cos(n\varphi) + \frac{x-y}{2\sqrt{-D}}
\cos(n\varphi) i  
- \frac12 \sin(n\varphi) i + \frac{x-y}{2\sqrt{-D}}\sin(n\varphi) \right. \\
\left. + \frac12 \cos(n\varphi) - \frac{x-y}{2\sqrt{-D}} \cos(n\varphi) i 
+ \frac12 \sin(n\varphi) i + \frac{x-y}{2\sqrt{-D}} \sin(n\varphi) \right).
\end{multline*}
This shows that the imaginary parts cancel out. This proves Eq.~\eqref{eq:15}.
\par In the third case, the solution to Eq.~\eqref{eq:4} is $\poly{P_n} = (c_1
+ c_2 n ) r^n$ where $r = \frac{x+y}{2}$ is the unique root of
Eq.~\eqref{eq:16} and $c_1,c_2\in\mathbb{R}$ are determined by
Eq.~\eqref{eq:14}.  If $x+y=0$, then $\poly{P_n}=0$. Thus, in this case,
Eq.~\eqref{eq:8} holds.  If $x+y\neq 0$, then by Eq.~\eqref{eq:14}, $c_1 = 1$
and $c_2 = \frac{x-y}{x+y}$; hence,
\[ \poly{P_n} = \left( 1 + \frac{x-y}{x+y} n \right) \left( \frac{x+y}{2}
\right)^n, \] which coincides with Eq.~\eqref{eq:8}.  This completes the
proof.
\end{proof}

For any $n\in\mathbb{N}$, we use $C_n$ to denote the connected 2-regular graph
with $n$ vertices. We adopt the convention that $C_0$ is the empty graph.  By
Eq.~\eqref{eq:rek} we have
\begin{align}
\poly{C_1} & = x + xy + z, \label{eq:12} \\
\poly{C_2} 
& = x^2 + 2xy + 2z + xy^2 + yz. \label{eq:7}
\end{align}
The following theorem generalizes Eqs.~\eqref{eq:12} and \eqref{eq:7} to
cycles of any finite length.

\begin{theorem}
\label{thm2}
Let $n\in\mathbb{N}$ and $x,y,z\in\mathbb{R}$.  Let $D$ and $\varphi$ be
defined as in Eq.~\eqref{eq:3} resp.\ \eqref{eq:17}.  If $z \ge - \left(
\frac{x+y}{2} \right)^2$, then
\begin{gather}
\label{eq:9}
\poly{C_n} =  {\left(\frac{x + y - \sqrt{D}}{2}\right)}^{n}
+ {\left(\frac{x + y + \sqrt{D}}{2}\right)}^{n} +y^{n-1}(xy-y+z).
\end{gather}
If $z \le - \left( \frac{x+y}{2} \right)^2$, then
\begin{gather}
\label{eq:19}
\poly{C_n} = 2(-z)^{n/2} \cos(n\varphi) + y^{n-1} (xy-y+z).
\end{gather}
\end{theorem}

\begin{proof}
For \mbox{$n=1, 2$} the theorem agrees under both conditions on $z$ with
Eqs.~\eqref{eq:12} and \eqref{eq:7}.  This is easy to see for $z \ge - \left(
\frac{x+y}{2} \right)^2$, while for $z \le - \left( \frac{x+y}{2} \right)^2$ the
identities $\cos(\arctan(t))$ $= 1/\sqrt{1+t^2}$ and $\cos(\alpha) =
2\cos^2(\alpha)-1$ reveal the coincidence.
\par
For the rest of this proof, we assume $n\ge 3$.  We may further assume that $z
\neq - \left( \frac{x+y}{2} \right)^2$ as the remaining case follows for
reasons of continuity by taking limits on both sides of Eqs.~\eqref{eq:9} and
\eqref{eq:19} as $z \downarrow - \left( \frac{x+y}{2} \right)^2$ resp.\
$z\uparrow - \left( \frac{x+y}{2} \right)^2$. 
By Eq.~\eqref{eq:rek} we have the non-homogeneous recurrence
\begin{gather*}
\poly{C_n} = \poly{P_n} + y\poly{C_{n-1}} + z \poly{P_{n-2}} \quad (n\ge 3)
\end{gather*}
with initial condition as in Eq.~\eqref{eq:7}.  Iterating this recurrence
gives
\begin{align}
\poly{C_n} & = \sum_{j=0}^{n-2} y^j \Big( \poly{P_{n-j}} + z \poly{P_{n-j-2}}
\Big) + y^{n-1} (x+xy+z)
\notag \\
& =  \poly{P_n} + y \poly{P_{n-1}} +  (y^2+z) \sum_{j=0}^{n-4} y^j
\poly{P_{n-j-2}} \notag \\
& \qquad\qquad\qquad\qquad\qquad\qquad +  y^{n-3}\left( xz + yz + xy^2 + xy^3 +y^2 z\right)  .
\label{eq:11}
\end{align}
Using Eq.~\eqref{eq:5} with $c_1,r_1,c_2,r_2$ from Eqs.~\eqref{eq:23} and
\eqref{eq:27} in the preceding proof, the sum on the right-hand side of
Eq.~\eqref{eq:11} can be written as
\begin{align*}
\sum_{j=0}^{n-4} y^j \poly{P_{n-j-2}} 
& = \sum_{j=0}^{n-4} y^j \left( c_1 r_1^{n-j-2} + c_2 r_2^{n-j-2} \right)
\notag \\
& = c_1 r_1^{n-2} \sum_{j=0}^{n-4} \left( \frac{y}{r_1} \right)^j + c_2
r_2^{n-2} \sum_{j=0}^{n-4} \left( \frac{y}{r_2} \right)^j \,  .
\end{align*}
\par
If $z\neq -xy$, then $y\neq r_1$ and $y\neq r_2$.  In this case, by applying
the formula for finite geometric series the preceding equation simplifies to
\begin{align*}
\sum_{j=0}^{n-4} y^j \poly{P_{n-j-2}} & = c_1 r_1^2
\frac{r_1^{n-3}-y^{n-3}}{r_1-y} + c_2 r_2^2 \frac{r_2^{n-3}-y^{n-3}}{r_2-y} .
\end{align*}
Substituting this latter expression into Eq.~\eqref{eq:11} and taking into
account that $r_1$ and $r_2$ are given as in Eqs.~\eqref{eq:23} and
\eqref{eq:27} leads to
\begin{align}
\poly{C_n} & = c_1 r_1^n +c_2 r_2^n + y(c_1 r_1^{n-1} + c_2 r_2^{n-1}) \notag \\ 
& \qquad +\,
(y^2+z)\left( c_1 r_1^2 \frac{r_1^{n-3}-y^{n-3}}{r_1-y} + c_2 r_2^2
\frac{r_2^{n-3}-y^{n-3}}{r_2-y} \right) \notag \\
& \qquad +\, 
y^{n-3}\left( xz +yz +xy^2 +xy^3 +y^2z \right) \notag \\
& = r_1^n + r_2^n +y^{n-1} (xy - y + z) \, ,  \label{sec:clos-form-expans-1}
\end{align}
where the last equality follows by substituting $c_1 = -
\frac{r_1-y}{\sqrt{D}}$, $c_2 = \frac{r_2-y}{\sqrt{D}}$, $\sqrt{D} =
-\frac{r_1^2-r_2^2}{x+y}$, and rearranging and cancelling terms (note that
$\sqrt{D} = i\sqrt{-D}$ if $D<0$).  Now, for $z>-\left(\frac{x+y}{2}\right)^2$
Eq.~\eqref{eq:9} follows from Eqs.~\eqref{sec:clos-form-expans-1} and
\eqref{eq:23}, whereas for $z<-\left(\frac{x+y}{2}\right)^2$ Eq.~\eqref{eq:19}
follows from Eqs.~\eqref{sec:clos-form-expans-1} and \eqref{eq:25} after
cancelling out the imaginary parts, in analogy to the proof of
Theorem~\ref{thm1}. \par
If $z=-xy$, then $z>-\left( \frac{x+y}{2} \right)^2$.  In this remaining case,
the result follows for reasons of continuity by taking limits on both sides of
Eq.~\eqref{eq:9} as $z\downarrow -xy$.
\end{proof}

\begin{remark}
For $z = - \left( \frac{x+y}{2} \right)^2$, Eqs.~\eqref{eq:9} and \eqref{eq:19}
coincide. In this case,
\[ \poly{C_n} = 2 \left( \frac{x+y}{2} \right)^n - \frac{x^2-2xy+y^2+4y}{4} \,
y^{n-1} . \] Alternatively, this can be shown by combining Eqs.~\eqref{eq:8}
and \eqref{eq:11} and applying the formula for finite geometric series.
\end{remark}

\begin{remark}
The preceding closed-form expansions can also be proved by induction.
A~computer algebra system might be helpful. In \textsf{Sage} \cite{Stein}, for
instance, the following lines of code prove Eqs.~\eqref{eq:2} and \eqref{eq:9}
by induction on the number of vertices.
\begin{verbatim}
var("n x y z")
D = x^2+2*x*y+y^2+4*z
path = (sqrt(D)-x+y)/(2*sqrt(D))*((x+y-sqrt(D))/2)^n \
      +(sqrt(D)+x-y)/(2*sqrt(D))*((x+y+sqrt(D))/2)^n
cycle = ((x+y-sqrt(D))/2)^n+((x+y+sqrt(D))/2)^n+y^(n-1)*(x*y-y+z)
bool(path(n=0)==1 and path(n=1)==x \
      and (x+y)*path(n=n-1)+z*path(n=n-2)==path)
bool(cycle(n=1)==x+x*y+z and path+y*cycle(n=n-1)+z*path(n=n-2)==cycle)
\end{verbatim}
\end{remark}

We proceed with a corollary on the generating function of $\poly{G}$.

\begin{corollary}
\label{cor:clos-form-expans-3}
\begin{align*}
\sum_{n=0}^\infty \poly{P_n} t^n & = \frac{1-yt}{1-(x+y)t -zt^2} \, ,\\
\sum_{n=0}^\infty \poly{C_n} t^n & = \frac{1+zt^2}{1-(x+y)t-zt^2} +
\frac{(xy-y+z)t}{1-yt} \, .
\end{align*}
\end{corollary}

\begin{proof}
Corollary \ref{cor:clos-form-expans-3} is an immediate consequence of Theorem
\ref{thm1}, Theorem \ref{thm2} and the geometric series formula.
\end{proof}

\section*{Acknowledgement}

The author thanks Peter Tittmann for drawing his attention to the universal
edge elimination polynomial and its generating function.

\end{document}